\documentclass{article}%
\usepackage{graphicx}
\usepackage{amsmath,amssymb,amsfonts}%
\usepackage{theorem}%
\usepackage{color}%
\usepackage{hyperref}
\usepackage[font={small, it}]{caption}
\usepackage{caption}
\usepackage{subcaption}
\usepackage{tikz}

\theoremstyle{change}%
\newtheorem{definition}{Definition:}[section]%
\newtheorem{proposition}[definition]{Proposition:}%
\newtheorem{theorem}[definition]{Theorem:}%
\newtheorem{lemma}[definition]{Lemma:}%
\newtheorem{corollary}[definition]{Corollary:}%
{\theorembodyfont{\rmfamily}\newtheorem{remark}[definition]{Remark:}}%
{\theorembodyfont{\rmfamily}\newtheorem{example}[definition]{Example:}}%

\newenvironment{proof}
{{\bf Proof:}}
{\qquad \hspace*{\fill} $\Box$}%

\newcommand{\fs}{\mathfrak{s}}%
\newcommand{\rme}{\mathrm{e}}%

\newcommand{\CC}{\mathcal{C}}%
\newcommand{\OC}{\mathcal{O}}%
\newcommand{\UC}{\mathcal{U}}%
\newcommand{\XC}{\mathcal{X}}%

\newcommand{\T}{\mathbb{T}}%
\newcommand{\R}{\mathbb{R}}%
\newcommand{\Z}{\mathbb{Z}}%

\begin{document}

	\title{Minimal-time trajectories of a linear control system on a homogeneous space of the 2D Lie group}%
	\author{V\'{\i}ctor Ayala \\
		Instituto de Alta Investigaci\'{o}n\\
		Universidad de Tarapac\'{a}, Arica, Chile \\
vayala@academicos.uta.cl \\\\
Adriano Da Silva\\
		Departamento de Matem\'atica,\\Universidad de Tarapac\'a, Iquique, Chile \\
  adasilva@academicos.uta.cl \\\\		
 Maria Torreblanca\\
		Instituto de Matem\'{a}ticas\\
		Universidad de San Agustín de Arequipa,  Perú\\
  mtorreblancat@unsa.edu.pe 
	}
	\date{\today }
	\maketitle
	
	\begin{abstract}
 
 Through the Pontryagin maximum principle, we solve a minimal-time problem for a linear control system on a cylinder, considered as a homogeneous space of the solvable Lie group of dimension two. The main result explicitly shows the existence of an optimal trajectory connecting every couple of arbitrary states on the manifold. It also gives a way to calculate the corresponding minimal time. Finally, the system admits points with two distinct minimal-time trajectories connecting them.
		
	\end{abstract}
	
	
	
	\section{Introduction}

The celebrated Pontryagin maximum principle, a 1962 Lenin Prize in Russia, is a fundamental mathematical result in the optimal control theory. Given a control system $\Sigma _{M}$  on a differential manifold $M$, the principle establishes necessary conditions for a two-point boundary value problem with a cost to optimize on the associated Hamiltonian of the system. The minimal-time problem involves connecting two states $x,y$ in $M$ through $\Sigma _{M}$, solution at minimum time.

The previous analysis depends on the system's controllability property, which guarantees the existence of an $\Sigma _{M}$ trajectory connecting the given two states. In a more general setting, it depends on the presence of a control set, which is a subset of the manifold where controllability holds in its interior.

Let $S$ be a Lie group with algebra $\fs$, considered the set of left-invariant vector fields on $S$. A Linear Control System $\Sigma _{S}$, LCS, is determined by three objects: The drift, the control vectors, and the class $\mathcal{U}$ of the admissible control functions \cite{AT99}, \cite{Markus}. 
The drift is a linear vector field $\mathcal{X}$ which is determined by its flows $\left\{\mathcal{X}_{t}:t \in \mathbb{R}\right\}$, a $1$-parameter group of $\mathrm{Aut}(S),$ the Lie group of $S$-automorphisms. Furthermore, by the Jacobi identity of the bracket, the linear map $DY = -\mathrm{ad}(\mathcal{X})Y =\left [\mathcal{X} ,Y\right]$   is a derivation of $\fs$. Any control vector belongs to $\mathfrak{s}$. And $\mathcal{U}$ is the class of piecewise constant functions with values on a compact subset of an Euclidean space.

Next, consider a closed subgroup $H \subset S$. According to the general theory, any left-invariant vector field $Y \in \mathfrak{s}$ can be projected to the quotient manifold $H \setminus S$. The same is no longer valid for a linear vector field $\mathcal{X}$. However, if the Lie algebra $\mathfrak{h}$ of $H$ is $D$-invariant,  it is possible to project $\Sigma _{S}$ into a linear control system $\Sigma _{H \setminus S}$ on the homogeneous space $H \setminus S$.

Through the Pontryagin maximum principle, we solve a minimal-time problem for a linear control system on the cylinder $\mathcal{C}_H$, considered a homogeneous space of the solvable Lie group $S$ of dimension two. In this case, we consider $\Omega=[-\rho, \rho]$ for $\rho>0$. The time optimal Hamiltonian function of $\Sigma _{\mathcal{C}_H}$ allows us to compute the corresponding Hamiltonian equations of the system and apply the principle.
Our main result explicitly shows the minimal-time trajectories. It also gives a way to calculate the minimal-time connecting two arbitrary given points on the state manifold. It turns out that the control function ${\bf u}$ associated with an optimal trajectory always belongs 
to the boundary $\partial\Omega=\{-\rho, \rho\}$ with only two possibilities: It is constant or changes from $-\rho$ to $\rho$ (or vice-versa) once. 

Several reasons support the relevance of this study. Among others, homogeneous spaces form a family of differentiable manifolds of particular importance in mathematics, physics, and applications. The Euclidean, projective, affine, Grassmannian, and hyperbolic spaces belong to this class. See \cite{2 dim type} for the classification of $2$-dimensional homogeneous spaces, including the plane, sphere, cylinder, torus, hyperbolic plane, etc. Furthermore, the Jouan Equivalence Theorem \cite[Theorem 4]{Jouan} shows that LCSs on homogeneous spaces classify nonlinear systems on manifolds. Precisely, it states that any control affine system on a connected manifold, whose vector fields are complete and transitive, is diffeomorphic to an LCS on a homogeneous space. Equivalent systems share the same topological and differentiable properties. In particular, the controllability property and the control sets. 

The following references contain results for the class of LCS on arbitrary and low-dimensional Lie groups and homogeneous spaces,\cite{dim2, homogeneous, dim3, semisimple, finite center, AJ16, ASM, DJ14, DS, Jouan11}. 
For general results on Lie groups, their homogeneous spaces, and control systems from a geometric point of view, we mention \cite{Ag1, Bourbaki2, Jbook, PBGM62, Sachkov09}.

In the sequel, we briefly mention the main contents of the manuscript. In Section $2$, we establish the Pontryagin maximum principle for a minimal-time problem to a general control affine system $\Sigma_M$ on a connected manifold $M$. We introduce in coordinates a linear control system on the solvable $2$-dimensional group $S$. And we characterize the system that can be projected to the cylinder. Section $3$ uses the fact that the projected system we are interested in is controllable \cite{dim2}. Through the Pontryagin maximum principle, we analyze the existence of a minimal-time trajectory between two arbitrary points of the cylinder. We distinguish two cases: when the optimal trajectory associated with the optimal control is constant and equal to one of the boundary points of $\Omega$, or when it switches one time between the boundary points of $\Omega$. After some reductions to the problem, we show that there are two possible curves of minimal-time connecting the given points, and the result is obtained by calculating the minimal-time needed to connect two points in the same fiber. We end the section with an excellent example showing the minimal-time to connect two points in the same fiber is never obtained by the constant control equal to zero, which leaves fibers invariant. In Section $4$, we first introduce a couple of special functions whose behavior allows us to describe Theorem 4.3. This main result will enable us to choose between the two possible trajectories of minimal-time. It also gives an explicit way to calculate the minimal-time connecting two arbitrary given points in the cylinder. Finally, a beautiful consequence is obtained: there exist points in the cylinder that admit two minimal-time trajectories connecting them.

 \section{Preliminaries}

\subsection{Control systems and the Pontryagin maximum principle}

Let $X_0, X_1, \ldots, X_m$ be smooth vector fields on a smooth finite dimensional manifold $M$ (Here smooth means $\mathcal{C}^{\infty} $). A {\it control system} on $M$ is determined by the family of ODEs
\begin{flalign*}
   &&\dot{x}=X_0(x)+\sum_{i=1}^mu_iX_i(x), \;\;\;\;\mbox{ where }\;\;\;u=(u_1, \ldots, u_m)\in\Omega,  &&\hspace{-1cm}\left(\Sigma_M\right)
   \end{flalign*}

where $\Omega\subset\R^m$ is a nonempty subset. Denote by $\UC$ the set of piecewise functions with the image in $\Omega$. For any $\bf{u}\in\UC$ and $x\in M$ we denote by $t\mapsto \phi(t, x, {\bf u})$ the piecewise differentiable curve on $M$ satisfying $\Sigma_M$ with $\phi(0, x, {\bf u})=x$. The {\it positive orbit} of $\Sigma_M$ at $x\in M$ is defined as  
	$$\mathcal{O}^+(x)=\{\phi(t, x, {\bf u}), t\geq 0, {\bf u}\in\UC\}\;\;\mbox{ and }\;\;\mathcal{O}^-(x)=\{\phi(-t, x, {\bf u}), t\geq 0, {\bf u}\in\UC\}.$$
 The system $\Sigma_M$ is said to satisfy the Lie algebra rank condition (LARC) if the Lie algebra $\mathcal{L}$ generated by the vector
	fields $X_0, X_1, \ldots, X_m$, satisfies $\mathcal{L}(x)=T_{x}M$ for all $x\in M$. The system $\Sigma_M$ is {\it controllable} if any two points in $M$ can be connected through a solution of the system in positive-time, that is if $M=\OC^+(x)$ for all $x\in M$. 

Once controllability is assured, asking the optimal way to connect points is natural. One of these optimal problems is to find the minimal-time trajectory connecting two given points. The Pontryagin maximum principle (PMP), which we recall below, is essential in this direction. The reader can consult \cite{Ag1} for more on the subject.

 The associated Hamiltonian is defined as
$$\mathcal{H}_u(\lambda, x)=\left\langle\lambda_x, X_0(x)+\sum_{i=1}^m u_iX_i(x)\right\rangle, \hspace{.5cm}\mbox{ where }\hspace{.5cm} \lambda_x\in T_x^*M.$$

Let us assume that ${\bf u}(t), t\in [0, \tau]$, is a control function associated with a minimal- time trajectory. That is, the solution $x(t)=\phi(t, x_0, {\bf u})$ of $\Sigma_M$ is the one with minimal-time among all the possible solutions of $\Sigma_{M}$ steering
$x(0)=x_0$ to $x_1=\phi(t, x_0, {\bf u})$. Then, the Pontryagin maximum principle states the existence  a Lipschitzian curve $(\lambda(t), x(t))$ in
the cotangent space $T^*M$ of $M$ satisfying (see \cite[Theorem 12.1]{Ag1})
\begin{enumerate}
    \item $\lambda(t)\neq 0$ for all $t\in [0, \tau]$
    \item $\mathcal{H}_{{\bf u}(t)}(\lambda(t), x(t))=\max_{u\in\Omega}\mathcal{H}_u(\lambda(t), x(t))$
    \item $\mathcal{H}_{{\bf u}(t)}(\lambda(t), x(t))\geq 0$ for all $t\in [0, \tau]$
    \item $(\lambda(t), x(t))$ satisfies the equations
    $$\left\{\begin{array}{l}
        \frac{d}{dt} x(t)=\frac{\partial}{\partial \lambda} \mathcal{H}_{{\bf u}(t)}(\lambda(t), x(t))= X(x(t))+\sum_{i=1}^m{\bf u}_i(t)X_i(x(t)) \\
        \\
         \frac{d}{dt} \lambda(t)=-\frac{\partial}{\partial x} \mathcal{H}_{{\bf u}(t)}(\lambda(t), x(t))
    \end{array}\right.$$
\end{enumerate}

\subsection{Linear control systems on 2D homogeneous spaces}

Let $S$ be the 2D solvable Lie group $(\R^2, *)$ where the product is defined as
	$$(z_1, w_1)*(z_2, w_2)=(z_1+z_2, w_1+\rme^{z_1}w_2).$$
	The Lie algebra $\fs$ of $S$ is given by the 2D vectorial space $(\R^2, [\cdot, \cdot])$ where the bracket is given by  
	$$[(\alpha_1, \beta_1), (\alpha_2, \beta_2)]=(0, \alpha_1\beta_2-\alpha_2\beta_1).$$
	With the previous setup, a left-invariant vector field and a linear vector field on $S$, respectively, read
	$$Y(z, w)=(\alpha, \rme^z\beta)\;\;\;\mbox{ and }\;\;\;\XC(z, w)=(0, bw+(\rme^z-1)a),$$
	where $(\alpha, \beta), (a, b)\in\R^2$. Hence, a {\it linear control system (LCS)} on $S$ is defined by the family of ODE's as follows  
   
   \begin{flalign*}
   &&\left\{
   \begin{array}{l}
   \dot{z}=u\alpha\\
   \dot{w}=bw+(\rme^z-1)a+u\rme^z\beta
   \end{array}\right., \;\;\;\;\mbox{ where }\;\;\;u\in\Omega,  &&\hspace{-1cm}\left(\Sigma_{S}\right)
   \end{flalign*}
  with $\Omega=[-\rho, \rho]$ for $\rho>0$. Moreover, $\Sigma_S$ satisfies the LARC if and only if $\alpha(a\alpha+b\beta)\neq 0$ (see \cite[Section 2.2]{dim2}). 

  From the general theory of LCSs, any control affine system satisfying the LARC, whose associated vector fields are complete and generate a finite-dimensional Lie algebra, is diffeomorphic to a linear system on a Lie group or to a control affine system on a homogeneous space induced by an LCS (see \cite[Theorem 4]{Jouan}). Due to this relevant fact, all the possible systems on the homogeneous of $S$ were classified in \cite{DSAyMA}. 

  The subgroup $H=\{0\}\times \Z\subset S$ is a closed subgroup and the homogeneous space 
  $H\setminus S$ is naturally identified with the horizontal cylinder $\CC_H:=\R\times \R/\Z$. Moreover, a system on $\CC_H$ is induced by an LCS on $S$ if and only if it has the form

\begin{flalign*}
&&\left\{\begin{array}{l}
\dot{x}=u\alpha\\
\dot{[y]}=(\rme^x-1)a+u\rme^x\beta
\end{array}\right., \;\;\;u\in\Omega  &&\hspace{-1cm}\left(\Sigma_{\CC_H}\right)
\end{flalign*}
Moreover, the Lie algebra generated by $X_0(x, [y]):=(0, (\rme^x-1)a)$ and $X_2(x, [y]):=(\alpha, \rme^x\beta)$ is given by 
$$\mathcal{L}_{\CC_H}(x, y)=\mathrm{span}\{(\alpha, \rme^x\beta), (0, a\alpha\rme^x)\},$$
implying that $\Sigma_{\CC_H}$ satisfies the LARC if and only if $a\alpha\neq 0$. Therefore, if $\Sigma_{\CC_H}$ satisfies the LARC, the map 
$$\varphi:\CC_H\rightarrow\CC_H, \hspace{1cm}\varphi(z, [y])=\left(x, \left[\frac{1}{a}\left(w-\frac{\beta}{\alpha}\rme^x\right)\right]\right),$$
is a diffeomorphism that conjugates $\Sigma_{\CC_H}$ to the system 
\begin{align}
 &\left\{
\begin{array}{rcl}
\dot{x}&=&u \nonumber\\
\dot{[y]} &=&(e^x-1)
\end{array}\right.,\hspace{2cm} \alpha u\in \Omega, \hspace{4cm} 
\end{align}

 \section{Minimal-time trajectories on $\CC_H$}

 In this section, we analyze the existence of minimal-time trajectories of an induced LCS on the horizontal cylinder $\CC_H$. 

 oBy the previous section, it is enough to consider the system. 

 \begin{align}
 &\left\{
\begin{array}{rcl}
\dot{x}&=&u \nonumber\\
\dot{[y]} &=&(e^x-1)
\end{array}\right.,\hspace{2cm} u\in \Omega=[-\rho, \rho], \hspace{4cm} \left(\Sigma_{\mathcal{C}_H}\right)
\end{align}
for some $\rho>0$. The solutions of $\Sigma_{\CC_H}$ starting at $P=(x, [y])$ for $u\in\Omega$ are given explicitly by 
$$\phi(t, P, u)=(x, \left[y +t(e^{x}-1)\right]), \hspace{.5cm}\mbox{ if }\hspace{.5cm}u=0,$$
and
\begin{equation}
\label{solutions}
\phi(t, P, u)=\left(x+ u t, \left[y +\frac{\rme^x}{u}(e^{u t}-1)-t\right]\right), \hspace{.5cm}\mbox{ if }\hspace{.5cm}u\neq 0.
\end{equation}
Since,
$$\lim_{t\rightarrow+\infty}\phi_1(t, P, \pm\rho)=\lim_{t\rightarrow+\infty}(x+u t)=\pm\rho\cdot+\infty,$$
we get that $\Sigma_{\CC_H}$ is controllable (see \cite[Theorem 4.3]{DSAyMA}). As a consequence, one can use the PMP to analyze the existence of minimal-time trajectories between any two given points of $\CC_H$.

Now, the fact that $\CC_H=\R\times\R/\Z$ implies that 
\[T_{P}^*\CC_H=\left({\mathbb R}^2\right)^*\cong {\mathbb R}^2,\]
and hence, any element $\lambda \in T_{P}^*\CC_H$ is identified with a vector of ${\mathbb R}^2$. For $\lambda=(p, q)\in T^*_P\CC_H$, the Hamiltonian of $\Sigma_{\CC_H}$, for the minimal-time problem, is given by
$$\mathcal{H}_{u}(P, \lambda)=\langle (u, (e^x-1)), (p, q)\rangle=pu+q(e^x-1).$$
Moreover, the Hamiltonian equations of $\Sigma_{\CC_H}$ are
\begin{equation}\left\{
\begin{array}{rcl}
\dot{p}&=&-qe^x\nonumber\\
 \dot{q} &=&0 \nonumber
\end{array}\right. \\
\end{equation}
Therefore, a control ${\bf u}\in\UC$ associated with a minimal-time trajectory $t\mapsto \phi(t, P_0, {\bf u})$
connecting $P_0$ and $P_1$ 
gives rise to a curve 
$$\lambda(t)=(p(t), q(t))\in T_{\phi(t, P_0, {\bf u})}\CC_H, \hspace{.5cm}\mbox{ with }\hspace{.5cm} \lambda(t)\neq 0 \hspace{.5cm}\mathrm{ a.e.}$$

From the fact that $\lambda(t)$ satisfy the Hamiltonian equations, we obtain that
$$\dot{q}=0\hspace{.5cm}\implies\hspace{.5cm} q(t)\equiv q_0\hspace{.5cm}\mbox{ and so }\hspace{.5cm}\dot{p}(t)=-q_0\mathrm{e}^{x(t)},$$
showing that 
$$p(t)\equiv p_0\neq 0\hspace{.5cm}\mbox{ when }\hspace{.5cm}q_0=0\hspace{.5cm}\mbox{ and }\hspace{.5cm}\dot{p}(t)\neq 0\hspace{.5cm}\mbox{ when }\hspace{.5cm} q_0\neq0.$$
Therefore, the function $p(t)$ changes sign at most one time. On the other hand, the fact that 
\begin{equation}
\mathcal{H}_{{\bf u}(t)}(\phi(t, P_0, {\bf u}), \lambda(t))=\max_{u\in\Omega}\left\{q(t)(\rme^{x(t)}-1)+u p(t)\right\},
\end{equation}
implies that ${\bf u}(t)\in\{-\rho, \rho\}$ a.e. and that ${\bf u}$ changes from $-\rho$ to $\rho$ (or vice-versa) at most one time. 

\subsection{The possible minimal-time trajectories}

In this section, we analyze the possible minimal-time trajectories. Let us fix $P_0, P_1\in\CC_H$ and assume the existence of a minimal-time trajectory
$$t\in [0, T]\mapsto\phi(t, P_0, {\bf u})\hspace{.5cm}\mbox{ with }\hspace{.5cm}\phi(T, P_0, {\bf u})=P_1.$$ 

By the previous section, the control function ${\bf u}$ associated with one such trajectory satisfies ${\bf u}(t)\in\{-\rho, \rho\}$ with ${\bf u}$ constant or with ${\bf u}$ changing one time between $-\rho$ and $\rho$.

\subsubsection{${\bf u}$ has no switches}

The next lemma gives us the minimal-time needed to connect two fibers $\mathbb{T}_{x}:=\{x\}\times \R/\Z$ with a constant control function.

\begin{lemma}
\label{fibers}
Let $u\in\Omega$ and assume that $x_0\neq x_1$. For any $u\in\Omega$ such that $u\cdot(x_1-x_0)>0$ it holds that
$$\phi(T, \T_{x_0},u)=\T_{x_1}\hspace{.5cm}\mbox{ where } \hspace{.5cm}T=\frac{x_1-x_0}{u}.$$
In particular, the minimal-time to connect two distinct fibers by a constant control is $\frac{|x_1-x_0|}{\rho}.$

\end{lemma}

\begin{proof}
    Let $P_i\in \T_{x_i}$ and $u\in\Omega$ such that $u\cdot(x_1-x_0)>0$. Then, 
    $$T=\frac{x_1-x_0}{u}>0\hspace{.5cm}\mbox{ and }\hspace{.5cm}\phi_1(T, P_0, u)=x_0+u T=x_0+u\frac{x_1-x_0}{u}=x_1,$$
    showing that $\phi(T, \T_{x_0}, u)\subset \T_{x_1}$. In the same way, it holds that 
    $$\phi(-T, \T_{x_1}, u)\subset\T_{x_0} \hspace{.5cm}\implies\hspace{.5cm}\phi(T, \T_{x_0}, u)=\T_{x_1}.$$
    To conclude the lemma, let us note that 
    $$\frac{|x_1-x_0|}{\rho}=\min\left\{\frac{x_1-x_0}{u}, \;\;u\in\Omega,\;\;u\cdot(x_1-x_0)>0\right\}$$
\end{proof}

\subsubsection{${\bf u}$ has one switch}

Let us now assume that ${\bf u}$ changes one time between $-\rho$ and $\rho$. We have two possibilities:
$$P_1=\phi(t_1, \phi(t_0, P_0, -\rho), \rho)\hspace{.5cm}\mbox{ and }\hspace{.5cm}P_1=\phi(s_1, \phi(s_0, P_0, \rho), -\rho),$$
with associated times $t_0+t_1$ and $s_0+s_1$, respectively. Therefore, to obtain the minimal-time trajectory connecting $P_0$ and $P_1$ 
in this case, it is enough to analyze the existence of $t_0, t_1, s_0, s_1$ and compare $t_0+t_1$ and $s_0+s_1$.

Using equations (\ref{solutions}) we get that 

$$
P_1=\phi(t_1, \phi(t_0, P_0, -\rho), \rho) \hspace{.5cm}\iff\hspace{.5cm}
\left\{
\begin{array}{l}
x_{0}- \rho t_{0}+\rho t_{1}=x_{1} \\
y_{0}-y_1-\frac{1}{\rho}e^{x_{0}}(\rme^{-\rho t_0}-1)-t_0+\frac{1}{\rho}e^{x_{0}-\rho t_0}(\rme^{\rho t_1}-1)-t_1\in\Z
\end{array}%
\right..
$$
Therefore, 
$$t_1=\frac{x_1-x_0}{\rho}+t_0\hspace{.5cm}\implies\hspace{.5cm}t_0+t_1=\frac{x_1-x_0}{\rho}+2t_0,$$
and 
$$y_{0}-y_1-\frac{1}{\rho}e^{x_{0}}(\rme^{-\rho t_0}-1)-t_0+\frac{1}{\rho}e^{x_{0}-\rho t_0}(\rme^{-\rho t_1}-1)-t_1=y_{0}-y_1+\frac{e^{x_{0}}}{\rho}\left(1-2e^{-\rho t_0%
}\right)+\frac{e^{x_1}}{\rho}-\left(\frac{x_1-x_0}{\rho}+2t_0\right).$$

Analogously, one gets 
$$
P_1=\phi(s_1, \phi(s_0, P_0, \rho), -\rho) \hspace{.5cm}\iff\hspace{.5cm}
\left\{
\begin{array}{l}
x_{0}+\rho s_{0}-\rho s_{1}=x_{1} \\
y_{0}-y_1+\frac{1}{\rho}e^{x_{0}}(\rme^{\rho s_0}-1)-s_0-\frac{1}{\rho}e^{x_{0}+\rho s_0}(\rme^{-\rho s_1}-1)-s_1\in\Z
\end{array}%
\right.,
$$
and hence, 
$$s_0=\frac{x_1-x_0}{\rho}+2s_1\hspace{.5cm}\implies\hspace{.5cm}s_0+s_1=\frac{x_1-x_0}{\rho}+2s_1,$$
and 
$$y_{0}-y_1+\frac{1}{\rho}e^{x_1}\left(2\rme^{\rho s_1}-1\right)-\frac{e^{x_0}}{\rho}-\left(\frac{x_1-x_0}{\rho}+2s_1\right)\in\Z.$$

\bigskip

\begin{proposition}

For any $P_0, P_1\in\mathcal{C}_H$ there exists $t_0, t_1, s_0, s_1\geq 0$ such that 
$$P_1=\phi(t_1, \phi(t_0, P_0, -\rho), \rho)\hspace{.5cm}\mbox{ and }\hspace{.5cm}P_1=\phi(s_1, \phi(s_0, P_0, \rho), -\rho).$$

\end{proposition}

\begin{proof}
    By the previous calculations, we only have to prove the existence of $t_0, s_1\geq 0$, satisfying
    $$y_{0}-y_1+\frac{e^{x_{0}}}{\rho}\left(1-2e^{-\rho t_0%
}\right)+\frac{e^{x_1}}{\rho}-\left(\frac{x_1-x_0}{\rho}+2t_0\right)\in\Z\hspace{.5cm}\mbox{ and }\hspace{.5cm} \frac{x_1-x_0}{\rho}+2t_0\geq 0,$$
and 
$$y_{0}-y_1+\frac{1}{\rho}e^{x_1}\left(2\rme^{\rho s_1}-1\right)-\frac{e^{x_0}}{\rho}-\left(\frac{x_1-x_0}{\rho}+2s_1\right)\in\Z \hspace{.5cm}\mbox{ and }\hspace{.5cm} \frac{x_1-x_0}{\rho}+2s_1\geq 0.$$

However, 
$$\lim_{t\rightarrow+\infty}\left(y_{0}-y_1+\frac{e^{x_{0}}}{\rho}\left(1-2e^{-\rho t}\right)+\frac{e^{x_1}}{\rho}-\left(\frac{x_1-x_0}{\rho}+2t\right)\right)=-\infty,$$
and 
$$\lim_{s\rightarrow+\infty}\left(y_{0}-y_1+\frac{1}{\rho}e^{x_1}\left(2\rme^{\rho s}-1\right)-\frac{e^{x_0}}{\rho}-\left(\frac{x_1-x_0}{\rho}+2s\right)\right)=+\infty,$$
certainly implies the existence of $t_0, s_1\geq 0$ satisfying the previous, proving the result.
\end{proof}

\subsection{Reductions and remarks}

This section shows that analyzing points in the same fiber is enough to obtain minimal-time.

In order to do that, let us define the functions  $F, G:[0, +\infty)\times\CC_H^2\rightarrow\R,$ by
$$F(t, P_0, P_1)=y_{0}-y_1+\frac{e^{x_{0}}}{\rho}\left(1-2e^{-\rho t%
}\right)+\frac{e^{x_1}}{\rho}-\left(\frac{x_1-x_0}{\rho}+2t\right)$$
and 
$$G(s, P_0, P_1)=y_{0}-y_1+\frac{1}{\rho}e^{x_1}\left(2\rme^{\rho s}-1\right)-\frac{e^{x_0}}{\rho}-\left(\frac{x_1-x_0}{\rho}+2s\right),$$
and consider
$$t_0=t_0(P_0, P_1):=\min\left\{t\geq 0, F(t, P_0, P_1)\in\Z\hspace{.3cm}\mbox{ and } \hspace{.3cm}\frac{x_1-x_0}{\rho}+2t\geq 0\right\}$$
$$s_1=s_1(P_0, P_1):=\min\left\{s\geq 0, G(s, P_0, P_1)\in\Z\hspace{.3cm}\mbox{ and } \hspace{.3cm}\frac{x_1-x_0}{\rho}+2s\geq 0\right\}.$$

The results in the previous section, together with the Pontryagin maximum principle, imply that the minimal-time necessary to connect two points $P_0, P_1$ is one of the following
$$T=\frac{x_1-x_0}{\rho}+2t_0\hspace{.5cm}\mbox{ or }\hspace{.5cm}S=\frac{x_1-x_0}{\rho}+2s_1.$$

\begin{itemize}
    \item[\bf (A)] The minimal-time does not depend on the representatives $[y_i]\in \T_{x_i}, i=0, 1$. 
    
    In fact, since 
    $$F(t, P_0, P_1)\in\Z\hspace{.5cm}\iff\hspace{.5cm}F(t, P_0, P_1)+n\in\Z, \hspace{.5cm}\forall n\in\Z,$$
    and the same holds for the function $G$; we have the independence of the representatives. 
\end{itemize}

Despite the simplicity, the previous remark will significantly help the analysis of the minimal-time. It will allow us to choose representatives, which simplifies the calculations.

\begin{itemize}
    \item[\bf(B)] The first reduction we can make is assuming that $x_0\leq x_1$. 
    
    The case $x_0>x_1$ can be recovered from analogous calculations by changing $t_0$ by $t_1$ and $s_1$ by $s_0$. 

\end{itemize}

The previous assumption only guarantees that the minimal-time to connect $P_0$ and $P_1$ is equal to the sum of the minimal-time needed to connect the fibers $\T_{x_0}$ and $\T_{x_1}$, with the minimal-time needed to connect two points in the same fiber. 
    
    In fact, under such assumption the equality $\phi(t_1, \phi(t_0, P_0, -\rho), \rho)$ assures that one leaves $P_0$ and reaches the point $\phi(t_0, P_0, -\rho)$ whose first coordinate satisfies
$$\phi_1(t_0, P_0, -\rho)=x_0-\rho t_0\leq x_0,$$
and then go to the point $P_1$ using the control $\rho$. As a consequence, the curve
\begin{equation}
    \label{minimalT}
    t\in[0, T]\mapsto\phi(t, P_0, {\bf u}^{-, +})\in\CC_H, \hspace{.5cm}\mbox{ where }\hspace{.5cm}{\bf u}^{-, +}(t)=\left\{\begin{array}{cc}
    -\rho & \mbox{ if } t\in [0, t_0] \\
   \rho  & \mbox{ if } t\in (t_0, t_0+t_1]
\end{array}\right.,
\end{equation}
intersects the fiber $\T_{x_0}$ at a point $P_0'$, in positive-time, before reaching the point $P_1$ (see Figure \ref{figd} left-side). Analogously, the curve  

\begin{equation}
    \label{minimalS}
    s\in[0, S]\mapsto\phi(s, P_0, {\bf u}^{+, -})\in\CC_H, \hspace{.5cm}\mbox{ where }\hspace{.5cm}{\bf u}^{+, -}(t)=\left\{\begin{array}{cc}
    \rho & \mbox{ if } s\in [0, s_0] \\
   -\rho  & \mbox{ if } s\in (s_0, s_0+s_1]
\end{array}\right.,
\end{equation}
intersects the fiber $\T_{x_1}$ at a point $P_1'$, in positive-time, before reaching the point $P_1$ (see Figure \ref{figd} right-side).

\begin{figure}[htbp!]
\begin{center}
\includegraphics[scale=.8]{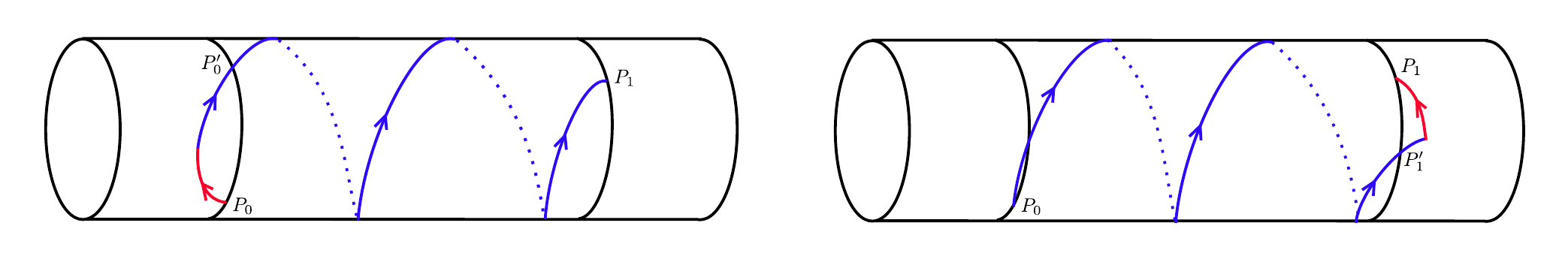}~\\
\caption{Trajectories associated with ${\bf u}^{-, +}$ and ${\bf u}^{+, -}$}
\label{figd}
\end{center}
\end{figure}

As a consequence, the minimal-time needed to connect $P_0$ and $P_1$ is determined by the minimum 
$$\min\{t_0(P_0, P_0'), s_1(P_1', P_1)\},$$
or equivalently, we have to analyze which of the functions 
$$F(t, P_0, P_0')\hspace{.5cm}\mbox{ and  }\hspace{.5cm}G(s, P_1', P_1),$$
arrives first in $\Z$. Since $P_i, P'_i$ belongs to the same fiber $\T_{x_i}, i=0,1$, the previous functions are given by 
\begin{equation}
\label{eq2}
\begin{aligned}
&F(t, P_0, P'_0)=y_{0}-y'_0+2\frac{e^{x_{0}}}{\rho}\left(1-e^{-\rho t%
}\right)-2t,&\hspace{4cm}\\
\mbox{ and }\hspace{2.3cm}& &\\
 &G(s, P'_1, P_1)=y'_{1}-y_1+2\frac{e^{x_1}}{\rho}\left(\rme^{\rho s}-1\right)-2s.&\hspace{4cm}
\end{aligned}
\end{equation}

\begin{itemize}
    \item[\bf (C)] The functions $F, G$ in (\ref{eq2}) only depends on $x_0, x_1$ and the difference $y_0-y_0'$.
\end{itemize}

In fact, the points $P'_0$ and $P_1'$  satisfy
$$\phi(\tau, P'_0, \rho)=P_1\hspace{.5cm}\mbox{ and }\hspace{.5cm}\phi(\tau, P_0, \rho)=P'_1, \hspace{.5cm}\mbox{ for }\hspace{.5cm}\tau=\frac{x_1-x_0}{\rho}.$$
These relations together with $P_0'=(x_0, [y_0'])$ and $P_1'=(x_1, [y'_1])$
gives us that 
$$y'_0-y_1+\frac{e^{x_0}}{\rho}\left(e^{\rho\tau}-1\right)-\tau\in\Z\hspace{.5cm}\mbox{ and }\hspace{.5cm}y_0-y'_1+\frac{e^{x_0}}{\rho}\left(e^{\rho\tau}-1\right)-\tau\in\Z,
$$
and hence $(y_0-y'_0)\equiv(y'_1-y_1)$. Since by (A) the minimal-time is independent of the representatives of $[y_i], [y'_i], i=0, 1$, we obtain that
\begin{equation}
\label{eq3}
\begin{aligned}
&F(t, P_0, P'_0)=y_{0}-y_0'+2\frac{e^{x_{0}}}{\rho}(1-e^{-\rho t})-2t,&\hspace{4cm}\\
\mbox{ and }\hspace{2.3cm}& &\\
 &G(s, P'_1, P_1)=y_0-y_0'+2\frac{e^{x_1}}{\rho}(e^{\rho s}-1)-2s,&\hspace{4cm}
\end{aligned}
\end{equation}
with $y_0-y_0'\in [0, 1)$.

\begin{remark}

It is not hard to see that 
$$y_0-y_0'\equiv (y_0-y_1)+\frac{e^{x_0}}{\rho}\left(e^{\rho\tau}-1\right)-\tau\hspace{.5cm}\mbox{ where }\hspace{.5cm}\tau=\frac{x_1-x_0}{\rho},$$
is the minimal-time needed to go from the fiber $\T_{x_0}$ to the fiber $\T_{x_1}$. Therefore, 
$$y_0-y_0'\equiv 0\hspace{.5cm}\iff \hspace{.5cm}\phi(\tau, P_0, \rho)=P_1,$$ 
that is, $y_0-y_0'$ is the distance, in the fiber $\T_{x_1}$, of the points $\phi(\tau, P_0, \rho)$ and $P_1$.

\end{remark}

The following example compares the minimal-time needed to connect two distinct points in the same fiber $\T_{x}, x\neq 0$ in two ways. Precisely, we compare the time through the trivial control with the one obtained by switching the control once.

\begin{example}

Let $P_0=(x_0, [y_0])$ and $P_1=(x_0, [y_1])$ two distinct points in the same fiber. Assume $y_0-y_1\in (0, 1)$ and $x_0>0$. The function 
$$H(t, P_0, P_1)=y_0-y_1+t(\rme^{x_0}-1),$$
satisfies
$$H(t, P_0, P_1)\in\Z\hspace{.5cm}\iff\hspace{.5cm}\phi(t, P_0, 0)=P_1.$$
As a consequence, the minimal-time needed to connect the points $P_0$ and $P_1$ through the trivial control is 
$$\min\{t>0, \;H(t, P_0, P_1)\in\Z\}.$$
On the other hand, by simple calculations, we obtain 
$$G(t, P_0, P_1)-H(2t, P_0, P_1)=\frac{2\rme^{x_0}}{\rho}(\rme^{\rho t}-1-\rho t)>0.$$
Since we are assuming $x_0>1$,
$$t_0=\min\{t>0, \;H(t, P_0, P_1)\in\Z\}\hspace{.5cm}\implies\hspace{.5cm}H(t_0, P_0, P_1)=1\hspace{.5cm}\implies\hspace{.5cm}G\left(\frac{t_0}{2}, P_0, P_1\right)>1.$$
Hence, there exists $s\in\left(0, \frac{t_0}{2}\right)$ such that $G\left(s, P_0, P_1\right)=1$ Therefore, the time needed to connect $P_0$ to $P_1$ by the curve (\ref{minimalS}) is 
$S=2s<t_0$.
Analogously, if $x>0$, we can show that the time $T$ needed to connect $P_0$ to $P_1$ by the curve (\ref{minimalT}) satisfies $T<t_0$. In both cases, we conclude that, for $x_0\neq 0$, the minimal-time to connect different points in the fiber $\T_{x_0}$ is less than $t_0$.
\end{example}

\section{The main results}

This section shows all the possible minimal-time trajectories connecting two distinct points in the cylinder $\CC_H$. We start with a preliminary section analyzing the behavior of a pair of functions that appear naturally in the proof of the main results.

\subsection{A race to $\Z$}

Let $2a\in(0, 1), \rho>0$ and define the maps $F, G:\R^+\times\R\rightarrow\R$ as  
$$F(t, x)=a+\frac{e^{x}}{\rho}(1-e^{-\rho t})-t,\hspace{.5cm}
\mbox{ and }\hspace{.5cm}G(t, x)=a+\frac{e^{x}}{\rho}(e^{\rho t}-1)-t.$$
In this section, we are concerned with the following problem: Let $x_0\leq x_1$ and consider the smallest positive real numbers $t, s$ satisfying
$$F(t, x_0)\in\Z\hspace{.5cm}\mbox{ and }\hspace{.5cm}G(s, x_1)\in\Z.$$
What is the relative position of $t, s$ in the semi-axis $(0, +\infty)$?

 This analysis will be critical in the following sections since it provides information about the minimal-time needed to connect two distinct points in a cylinder.

We start by noticing that 
$$F(t, x)-G(t, x)=\frac{e^{x}}{\rho}(1-e^{-\rho t})-\frac{e^{x}}{\rho}(e^{\rho t}-1)=2\frac{e^{x}}{\rho}(1-\cosh(\rho t))<0.$$
Moreover, derivation on the first variable gives us that 
$$\frac{\partial F}{\partial t}(t, x)=\mathrm{e}^{x-\rho t}-1\hspace{.5cm}\mbox{ and }\hspace{.5cm}\frac{\partial G}{\partial t}(t, x)=\mathrm{e}^{x+\rho t}-1,$$
showing that the partial maps $F(\cdot, x)$ and $G(\cdot, x)$ have at most one (not simultaneously) critical point (see the left-hand side of Figure \ref{fige}). Moreover, it holds that 
$$\lim_{t\rightarrow+\infty}F(t, x)=-\infty\hspace{.5cm}\mbox{ and }\hspace{.5cm}\lim_{t\rightarrow+\infty}G(t, x)=+\infty\hspace{.5cm}.$$
Since we are assuming that $2a\in (0, 1)$, for any $x\in\R$ there exist unique $t, s\in \R^+$ such that 
\begin{equation}
    \label{01}
    F(t, x)=0\hspace{.5cm}\mbox{ and }\hspace{.5cm}G(s, x)=1.
\end{equation}

The next result reduces our analysis. 

\begin{proposition}
    Let $x_0\leq x_1$ and $t, s$ the smallest positive real numbers satisfying 
    $$F(t, x_0)\in\Z\hspace{.5cm}\mbox{ and }\hspace{.5cm}G(s, x_1)\in\Z.$$
    Then, the minimal-time $\tau=\min\{t, s\}$ satisfies (at least) one of the equations in (\ref{01}).
\end{proposition}

\begin{proof}
    Since the cases are analogous, let us assume that $\tau=t$. Now, the fact that $2a\in(0, 1)$ implies that $F(t, x_0)\in\{0, 1\}$. We have nothing to prove if $F(t, x_0)=0$. Let us then assume $F(t, x_0)=1$. In this case, $F(\cdot, x_0)$ necessarily admits a critical point, and hence $G(\cdot, x)$ is strictly increasing for any $x\geq x_0$. Therefore, 
    $$G(s, x_1)\in\Z\hspace{.5cm}\iff\hspace{.5cm} G(s, x_1)=1.$$
    On the other hand,
    $$x_1\geq x_0\hspace{.5cm}\implies\hspace{.5cm}G(t, x_1)\geq G(t, x_0)>F(t, x_0)=1.$$
    Since $G(0, x_1)=a<1$ there exists, by continuity $0<s'<t$ such that $G(s', x_1)=1$. However, the fact that $G(\cdot, x_1)$ is strictly increasing forces $s=s'<t$, contradicting the fact $t=\min\{t, s\}$, ending the proof.
\end{proof}

\bigskip

By the previous proposition, we only have to analyze possible values of $x\in\R$ and $t, s>0$ satisfying the equations in (\ref{01}). Note that
$$F(t, x)=0\hspace{.5cm}\iff\hspace{.5cm} \mathrm{e}^x=\rho\frac{t-a}{1-\rme^{-\rho t}}\hspace{.5cm}\mbox{ and }\hspace{.5cm} G(t, x)=1\hspace{.5cm}\iff\hspace{.5cm} \mathrm{e}^x=\rho\frac{1+t-a}{\rme^{\rho t}-1}.$$

As a consequence, the analysis of the functions
$$f(t):=\rho\frac{t-a}{1-\rme^{-\rho t}}\hspace{.5cm}\mbox{ and }\hspace{.5cm} g(t):=\rho\frac{1+t-a}{\rme^{\rho t}-1},$$
determine the points $(t, x)$ satisfying relations (\ref{01}). The following result describes the behavior of $f$ and $g$. The proof is straightforward, and we will omit it. Moreover, the properties of the functions presented in the following result are depicted on the right-hand side of Figure \ref{fige}.

\begin{lemma}
    For the functions $f, g$ defined previously, 
    it holds:
    \begin{enumerate}
        \item $\lim_{t\rightarrow 0^+}f(t)=-\infty$ and $\lim_{t\rightarrow+\infty}f(t)=+\infty$
        
        \item $\lim_{t\rightarrow 0^+}g(t)=+\infty$ and $\lim_{t\rightarrow+\infty}g(t)=0$

\item $f(t)=0$ iff $t=a$ and $g(t)>0$ for all $t>0$.

        \item $f'(t)>0$ and $g'(t)<0$ for all $t>0$.
    \end{enumerate}
    
\end{lemma}

\begin{figure}[htbp!]
\begin{center}
\includegraphics[scale=.5]{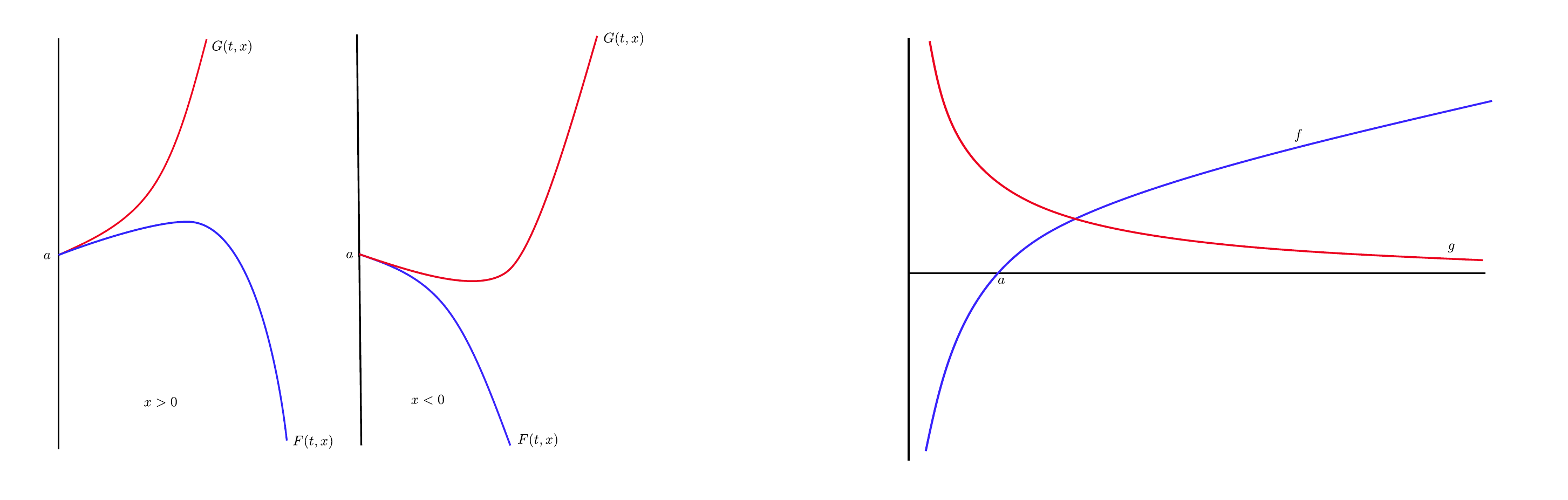}~\\
\caption{Behavior of $F, G$ and their associated functions $f, g$.}
\label{fige}
\end{center}
\end{figure}




    

\subsection{Minimal-time trajectories and their minimal-time}

Let us consider $P_0=(x_0, [y_0]), P_1=(x_1, [y_1])\in\CC_H$ with $x_0\leq x_1$ and consider 
$$\tau=\frac{x_1-x_0}{\rho}\hspace{.5cm}\mbox{ and }\hspace{.5cm} 2a\equiv (y_0-y_1)+\frac{1}{\rho}\rme^{x_0+\rho\tau}-\tau\in[0, 1).$$

By the previous sections, for $2a\in (0, 1)$, the functions $f:(a, +\infty)\rightarrow\R^+$ and $g:\R^+\rightarrow\R^+$ given by 
$$f(t):=\rho\frac{t-a}{1-\rme^{-\rho t}}\hspace{.5cm}\mbox{ and }\hspace{.5cm} g(t):=\rho\frac{1+t-a}{\rme^{\rho t}-1},$$
are bijections, with $g$ strictly decreasing and $f$ strictly increasing.

\begin{theorem}
\label{teo}
Under the previous notations, for the system $\Sigma_{\mathcal{C}_H}$ it holds:
\begin{enumerate}
    \item If $a=0$, then $t\mapsto \phi(t, P_0, \rho)$ is the minimal-time trajectory connecting $P_0$ and $P_1$ with associated minimal-time $\tau$.

    \item If $a\neq 0$ and $g^{-1}(\rme^{x_1})<f^{-1}(\rme^{x_0})$, then
    $$\left\{\begin{array}{lr}
       \phi(s, P_0, \rho)  &  \mbox{ for } s\in [0, \tau+g^{-1}(e^{x_1})] \\
        \phi\left(s, \phi\left(\tau+g^{-1}(e^{x_1}), P_0, \rho\right), -\rho\right) & \mbox{ for } s\in [0, g^{-1}(e^{x_1})] 
    \end{array}\right.,$$
    is the minimal-time trajectory connecting $P_0$ and $P_1$. The associated minimal-time is given by 
    $$S=\tau+2g^{-1}(e^{x_1}).$$

 \item If $a\neq 0$ and $g^{-1}(\rme^{x_1})>f^{-1}(\rme^{x_0})$, then 
    $$\left\{\begin{array}{lr}
       \phi(t, P_0, -\rho)  &  \mbox{ for } t\in [0, f^{-1}(\rme^{x_0})] \\
        \phi\left(t, \phi\left(f^{-1}(\rme^{x_0}), P_0, -\rho\right), \rho\right) & \mbox{ for } t\in [0, \tau+f^{-1}(\rme^{x_0})] 
    \end{array}\right.,$$
    is the minimal-time trajectory connecting $P_0$ and $P_1$. The associated minimal-time is given by 
    $$T=\tau+2f^{-1}(e^{x_0}).$$

 \item If $a\neq 0$ and $g^{-1}(\rme^{x_1})=f^{-1}(x_0)$
then both the curves in items 2. and 3. are minimal-time trajectories with associated minimal-time
$$\tau+2g^{-1}(x_1)=\tau+2f^{-1}(x_0)$$

\end{enumerate}

\end{theorem}

\begin{corollary}
The system $\Sigma_{\mathcal{C}_H}$ admits points $P_0, P_1$ with two distinct minimal-time trajectories connecting them.

\end{corollary}

\begin{remark}
    It is important to note that Theorem \ref{teo} gives us a way to explicitly calculate the minimal-time connecting two given points in $\CC_H$ through the inverse of the functions $f$ and $g$.
\end{remark}

\end{document}